\documentclass[final]{siamltex}

\usepackage{amssymb}
\usepackage{graphicx}
\usepackage{amsmath}
\usepackage{amssymb}
\usepackage[dvips]{color}
\usepackage{booktabs}
\usepackage{bbm}

\newtheorem{teo}{Theorem}
\newtheorem{lema}{Lemma}

\newtheorem{nota}{Remark}

\newtheorem{prop}{Proposition}

\DeclareMathOperator{\sign}{sign}

\title{Limit cycles for a class of eleventh $\mathbb{Z}_{12}-$equivariant systems without infinite critical points}
\author{Adrian C. Murza\thanks{Strada Grof Miko Imre Nr. 13, sc. B, ap. 11, Sf\^{a}ntu-Gheorghe, 520003, Romania. email: adrian.murza@fc.up.pt}}

\begin{document}
\maketitle

\begin{abstract}
We analyze the complex dynamics dynamics of a family of $\mathbb{Z}_{12}-$equivariant planar systems, by using their reduction to an Abel equation.
We derive conditions in the parameter space that allow uniqueness and hyperbolicity of a limit cycle surrounding either $1,~13$ or $25$ equilibria.
\end{abstract}

\maketitle



\section{Introduction and main results}\label{Introduction and main results}
Hilbert $XVI^{th}$ problem is one of the open question in mathematics and even in two dimension is far from being solved. The study of this problem became a new branch of analysis, based on the recent development of the theory of Golubitsky, Stewart and Schaeffer
in \cite{GS85,GS88}.

In this paper we analyze the $\mathbb{Z}_{12}-$equivariant system
\begin{equation}\label{main equation}
\dot{z}=pz^5\bar{z}^4+sz^6\bar{z}^5-\bar{z}^{11}=f(z),
\end{equation}
where where $z$ is complex, the time $t$ is real, while $p$ and $s$ are complex parameters ($p_1,p_2,s_1,s_2\in\mathbb{R}$).

The generalized form of the $\mathbb{Z}_q-$equivariant equation is
\[
 \dot{z}=zA(|z|^2)+B\bar{z}^{q-1}+O(|z|^{q+1}),
\]
where $A$ is a polynomial of degree $[(q-1)/2].$ The analysis of this class of
equations is carried out for example in \cite{arn, Che}, in the case of strong resonances, {\it i.e.} $q<4$ or weak ones $q>4.$ The case corresponding to $q=4$ is also analyzed in other articles, such as \cite{Che, Zegeling}.
Our aim is studying the phase portrait of system \eqref{main equation}. We devote especial attention to the existence, uniqueness and location of limit cycles surrounding $1,~13$ or $25$ equilibria. To reach our goal, we want to transform the system \eqref{main equation} into a scalar Abel equation and then reduce the study of the initial problem to the Abel equation.

We state now the main result of the article.
\begin{teo}\label{teorema principal}
Let us consider the equation \eqref{main equation} with $s_2>1,~p_2\neq0,$ and define the additional quantities:
\begin{eqnarray*}
\Sigma_A^-=\frac{p_2s_1s_2-\sqrt{p_2^2(s_1^2+s_2^2-1)}}{s_2^2-1},
\hspace{0.3cm}\Sigma_A^+=\frac{p_2s_1s_2+ \sqrt{p_2^2(s_1^2+s_2^2-1)}}{s_2^2-1}.
\end{eqnarray*}
Then, the following statements hold true:
\begin{itemize}
\item[(a)] If one of the constraints
\begin{equation*}
\hspace{2cm}(i)\hspace{0.3cm}p_1\notin\left(\Sigma_A^-,\Sigma_A^+\right),\hspace{1cm}(ii)\hspace{0.3cm}
p_1\notin\left(\frac{\Sigma_A^-}{2},\frac{\Sigma_A^+}{2}\right)
\end{equation*}
is satisfied, then equation \eqref{main equation} has at most one limit cycle surrounding the origin.
Moreover, if the mentioned limit cycles exists it is hyperbolic.
\item[(b)] There exist equations \eqref{main equation} under condition $(i)$ with exactly one hyperbolic limit cycle surrounding either $1$ or $13$ equilibria and equations \eqref{main equation} under condition $(ii)$ possessing exactly one hyperbolic limit cycle surrounding either $1,~13$ or $25$ equilibria.
\end{itemize}
\end{teo}

The paper is organized in the following manner. Section \ref{Preliminary results} is devoted to the statement of some preliminary results. In Section \ref{Analysis of the critical points} the analyze the equilibria. Finally, in Section \ref{Proof of Theorem $1.$} we give the detailed proof of the main theorem of the paper.

\section{Preliminary results}\label{Preliminary results}

From the pioneering work of Golubitsky, Schaeffer and Stewart \cite{GS85,GS88} a system of differential equations $dx/dt=f(x)$ is said to be $\Gamma-$equivariant if it commutes with the group action of $\Gamma,$ ie. $f(\gamma x)=\gamma f(x),~\forall\gamma\in\Gamma.$
Here $\Gamma=  \mathbb{Z}_{12}$ with the standard action on $\mathbb{C}$ generated by $\gamma_1=\exp(2\pi i/12)$ acting by complex multiplication.
Applying this result to equation \eqref{main equation} we have the following result.
\begin{prop}
Equation \eqref{main equation} is $\mathbb{Z}_{12}-$equivariant.
\end{prop}
\begin{proof}
Let $\gamma_k=\exp(2\pi ik/10),~k=0,\ldots,11.$ Then
\begin{equation*}
\begin{array}{l}
f(\gamma_kz)=f(\exp(2\pi ik/12)z)=(p_1+ip_2)\exp(10\pi ik/12)z^4\exp(-8\pi ik/12)\bar{z}^3+\\
\hspace{1cm}+(s_1+is_2)\exp( 12\pi ik/12   )z^5\exp(-10\pi ik/12)\bar{z}^4-\exp(-22\pi ik/12)\bar{z}^{11}=\\
\hspace{1cm}=\exp(2\pi ik/12)\left((p_1+ip_2)z^5\bar{z}^4+(s_1+is_2)z^6\bar{z}^5-\bar{z}^{11}\right)=\\
\hspace{1cm}=\exp(2\pi ik/12)f(z)= \gamma_k f(z).
\end{array}
\end{equation*}
\end{proof}

We obtain equation \eqref{main equation} by perturbing a Hamiltonian, whose expression is
$$\dot{z}=i(p_2+s_2z\bar{z})z^5\bar{z}^4-\bar{z}^{11}.$$
We express formally this fact in the following theorem.

\begin{teo}\label{teo-Ham}
The corresponding Hamiltonian equation to the $\mathbb{Z}_{12}-$equivariant equation \eqref{main equation}
is $$\dot{z}=i(p_2+s_2z\bar{z})z^5\bar{z}^4-\bar{z}^{11}.$$
\end{teo}
\begin{proof} The class of equations $\dot{z}=F(z,\bar{z})$ is Hamiltonian if $\displaystyle{\frac{\partial F}{\partial z}+\frac{\partial \bar{F}}{\partial \bar{z}} =0.}$ For
equation \eqref{main equation} we have
\begin{equation*}\label{hamilt2}
    \begin{array}{l}
    \displaystyle{\frac{\partial F}{\partial z}=5(p_1+ip_2)z^4\bar{z}^4+6(s_1+is_2)z^5\bar{z}^5}\\
    \\
    \displaystyle{\frac{\partial \bar{F}}{\partial \bar{z}}=5(p_1-ip_2)z^4\bar{z}^4+6(s_1-is_2)z^5\bar{z}^5}.
    \end{array}
\end{equation*}
Therefore equation \eqref{main equation} is Hamiltonian if and only if $p_1=s_1=0.$
\end{proof}

An important technique for studying the limit cycles of a two--dimensional $\mathbb{Z}_q$--equivariant ODE consists in reducing the initial ODE to a scalar equation of Abel type. In applying this technique to \eqref{main equation}, firstly we convert it from cartesian to polar coordinates.
\begin{lema}
The solutions of equation \eqref{main equation}
are equivalent to those of the polar system
\begin{equation}\label{polar1}
\left\{
    \begin{array}{l}
    \displaystyle{\dot{r}=2rp_1+2r^2\left(s_1-\cos (12\theta)\right)}\\
    \displaystyle{\dot{\theta}=p_2+r\left(s_2+\sin (12\theta)\right)}
    \end{array}.
    \right.
\end{equation}

\end{lema}
\begin{proof}
We perform the change of variables
\begin{equation*}
z=\sqrt{r}(\cos(\theta)+i\sin(\theta))
\end{equation*}
and obtain
\begin{equation}\label{polar121}
\left\{
    \begin{array}{l}
    \displaystyle{\dot{r}=2r^5(p_1+rs_1-r\cos (12\theta))}\\
    \displaystyle{\dot{\theta}=r^4(p_2+rs_2+\sin (12\theta))}.
    \end{array}.
    \right.
\end{equation}
Then we do the time rescaling, $\displaystyle{\frac{dt}{ds}=r^4,}$ to get the result.
\end{proof}

As a consequence, we have the following lemma.
\begin{lema}\label{lema abelian}
The analysis of non contractible solutions that satisfy $x(0)=x(2\pi)$ of the corresponding Abel equation affords the same information about the periodic orbits of equation \eqref{main equation} that surround the origin.

\begin{equation}\label{abelian0}
    \begin{array}{l}
    \displaystyle{\frac{dx}{d\theta}=A(\theta)x^3+B(\theta)x^2+C(\theta)x}
    \end{array}
\end{equation}
where
\begin{equation}\label{abelian1}
    \begin{array}{l}
    \displaystyle{A(\theta)=\frac{2}{s_2}\left(-s_1+p_1p_2s_2-p_2^2s_1+(p_1s_2-2p_2s_1)\sin(12\theta)\right)+}\\
    \hspace{1.3cm}+\displaystyle{\frac{2}{s_2}\left(s_2\sin(12\theta)+s_1\cos(12\theta)+p_2s_2)\cos(12\theta\right),}\\
    \\
    \displaystyle{B(\theta)=\frac{2}{s_2}\left(-p_1s_2-s_2\cos(12\theta)+2s_1\sin(12\theta)+2p_2s_1 \right),}\\
    \\
    \displaystyle{C(\theta)=-\frac{2s_1}{s_2}.}\\
    \\
    \end{array}
\end{equation}
\end{lema}
\begin{proof}
From equation \eqref{polar1} we obtain
\begin{equation*}
    \begin{array}{l}
    \displaystyle{\frac{dr}{d\theta}=\frac{\displaystyle{2r(p_1+rs_1+\cos (12\theta))}}{p_2+rs_2+\sin (12\theta)}.}
    \end{array}
\end{equation*}
Next we use a Cherkas-like technique \cite{Cherkas} to convert the scalar equation \eqref{polar1} into \eqref{abelian0}. More specifically, we use the variable change $x=\displaystyle{\frac{r}{p_2+rs_2+\sin (12\theta)}}$, to get the result. As the non contractible periodic orbits of equation \eqref{abelian0} cannot intersect the set $\{\dot{\theta}=0\},$ their analysis is the same as studying the limit cycles that surround the origin of equation \eqref{main equation}. This is explained in more detail in \cite{CGP}.
\end{proof}

Our aim is applying the methodology developed in \cite{Rafel1} to analyze
conditions for existence, location and unicity of the limit cycles surrounding $1,~13$ and $25$ critical points.

A precise method for proving the existence of a limit cycle is showing that, in the Poincar\'e compactification, infinity has no critical
points; moreover we have to prove that both origin and infinity have the same stability. We explicitly describe the set of
parameters which satisfies these conditions. The first step in our approach is to determine the stability of infinity. This is done in the next lemma.

\begin{lema}\label{lema characterization origin}
Apply the Poincar\'e compactification to equation \eqref{main equation}. The following statements hold true.
\begin{itemize}
\item [(i)]   Infinity has no equilibria if and only if  $|s_2|>1;$
\item [(ii)] For $s_2>1,$  infinity is attractor (resp. a repellor) if $s_1>0$ (resp. $s_1<0$) and it has the opposite character
for $s_2<-1.$
\end{itemize}
\end{lema}

\begin{proof}
Firstly, we perform the change of variable $R=1/r$ in system \eqref{polar1}. Then the parametrization by
$\displaystyle{\frac{dt}{ds}=R},$ yields the system
\begin{equation}\label{poincare}
\left\{
\begin{array}{l}
R'=\displaystyle{\frac{dR}{ds}=\displaystyle{-2s_1R-2R^2(p_1+\cos(12\theta))}},\\
\\
\theta'= \displaystyle{\frac{d\theta}{ds}=\displaystyle{R(s_2+\sin(12\theta))+p_2}}.
\end{array}
\right.
\end{equation}
This yields the the invariant set $\{R=0\},$ corresponding to infinity of system \eqref{polar1}. Hence, system \eqref{polar1} has no equilibria at infinity provided $|s_2|>1.$ We apply the method originally proposed by \cite{Lloyd}, to the system \eqref{poincare}. That is, we analyze the stability of $\{R=0\}$.
This is obtained from the sign of
\begin{equation*}
\displaystyle{\int_0^{2\pi}\frac{-2s_1}{s_2+\sin(12\theta)}d\theta}=\displaystyle{\frac{-\hbox{sgn}(s_2)4\pi s_1}{\sqrt{s_2^2-1}}}.
\end{equation*}
This completes the proof.
\end{proof}

\section{Analysis of the critical points}\label{Analysis of the critical points}
This section we explicitly determine the conditions that allow equation \eqref{polar1} to have one, thirteen or twenty--five equilibria.
Our interest in the origin is twofold. Firstly, we show (by simple inspection), that it is always an equilibrium. Secondly, we demonstrate that it is always \emph{monodromic}.

The main tool to carry out our proofs is represented by the generalized Lyapunov constants. To define those, we first recall the solution of the scalar equation\begin{equation}\label{def lyap constants1}
\displaystyle{\frac{dr}{d\theta}=\sum_{i=1}^\infty}R_i\left(\theta\right)r^i,
\end{equation}
where $R_i(\theta),~i\geqslant1$ are $T-$periodic functions. Next, we define the generalized Lyapunov constants; consider the solution of
\eqref{def lyap constants1} that for $\theta=0$ passes through $\rho.$ It can be described by
\begin{equation*}
\begin{array}{l}
\displaystyle{r(\theta,\rho)=\sum_{i=1}^{\infty}u_i(\theta)\rho^i}
\end{array}
\end{equation*}
with $u_1(0)=1,~u_k(0)=0, \forall~k\geqslant2.$
Therefore, the first return map of the solution above is written in terms of the series
\begin{equation*}
\begin{array}{l}
\displaystyle{\Pi(\rho)=\sum_{i=1}^{\infty}u_i(T)\rho^i}.
\end{array}
\end{equation*}
To analyzing the stability of a solution, the only significant term in the return map is the first nonvanishing term that makes it differ from the identity map. Therefore, is this term the one that determines the stability of
this solution. On the other side, let us take into account a family of equations that depends
on parameters, i.e. each of the terms $u_i(T)$ is a function of these parameters. We define
$k^{th}$ generalized Lyapunov constant $V_k=u_k(T)$ the value for the equation above
supposing $u_1(T) = 1, u_2(T) =\ldots,= u_{k-1}(T) = 0.$

\begin{lema}\label{lema Lyapunov}
The stability of the origin of system \eqref{polar1} is given by the sign of $p_1$ if $p_1\neq0$ and by the sign of $s_1$ if $p_1=0.$ In addition, the character of the origin is always monodromic.
\end{lema}

\begin{proof}
We begin with the proof of the last statement in Lemma \ref{lema Lyapunov}. To prove origin's monodromicity we firstly write system \eqref{polar1} in cartesian calculate. This yields system \eqref{cartesian}.
\begin{equation}\label{cartesian}
\begin{array}{l}
\dot{x}=P(x,y)=p_1 x^9 - x^{11} + s_1 x^{11} - p_2 x^8 y - s_2 x^{10} y + 4 p_1 x^7 y^2 +\\
 \hspace{2.8cm}55 x^9 y^2 + 5 s_1 x^9 y^2 - 4 p_2 x^6 y^3 - 5 s_2 x^8 y^3 +\\
 \hspace{2.8cm}6 p_1 x^5 y^4 - 330 x^7 y^4 + 10 s_1 x^7 y^4 - 6 p_2 x^4 y^5 -\\
 \hspace{2.8cm}10 s_2 x^6 y^5 + 4 p_1 x^3 y^6 + 462 x^5 y^6 + 10 s_1 x^5 y^6 -\\
 \hspace{2.8cm}4 p_2 x^2 y^7 - 10 s_2 x^4 y^7 + p_1 x y^8 - 165 x^3 y^8 +\\
 \hspace{2.8cm}5 s_1 x^3 y^8 - p_2 y^9 - 5 s_2 x^2 y^9 + 11 x y^{10} + s_1 x y^{10} -\\
 \hspace{2.8cm}s_2 y^{11},\\
\\
\dot{y}=Q(x,y)=p_2 x^9 + s_2 x^{11} + p_1 x^8 y + 11 x^{10} y + s_1 x^{10} y + 4 p_2 x^7 y^2 +\\
 \hspace{2.8cm}5 s_2 x^9 y^2 + 4 p_1 x^6 y^3 - 165 x^8 y^3 + 5 s_1 x^8 y^3 +\\
 \hspace{2.8cm}6 p_2 x^5 y^4 + 10 s_2 x^7 y^4 + 6 p_1 x^4 y^5 + 462 x^6 y^5 +\\
 \hspace{2.8cm}10 s_1 x^6 y^5 + 4 p_2 x^3 y^6 + 10 s_2 x^5 y^6 + 4 p_1 x^2 y^7 -\\
 \hspace{2.8cm}330 x^4 y^7 + 10 s_1 x^4 y^7 + p_2 x y^8 + 5 s_2 x^3 y^8 + p_1 y^9 +\\
 \hspace{2.8cm}55 x^2 y^9 + 5 s_1 x^2 y^9 + s_2 x y^{10} - y^{11} + s_1 y^{11}.\\
\end{array}
\end{equation}
Following \cite[Chapter IX]{Andronov}, we conclude that any solutions arriving at the origin are tangent to the directions $\theta$ that are the roots of the equation $r \dot \theta=R(x,y)=xQ(x,y)-yP(x,y)=p_2(x^2+y^2)^5$. It is easy to see, by simple inspection that this is always different than zero. In consequence that the origin is either a focus or a center.

We compute the first two Lyapunov constants in order to analyze the stability of the origin. For an Abel equation these may be expressed in the form
$$
V_1=\displaystyle{\exp\left(\int_0^{2\pi}C(\theta)d\theta\right)}-1,\qquad\qquad
V_2=\displaystyle{ \int_0^{ 2\pi }B(\theta)d\theta}.
$$
Combining these with equations \eqref{abelian1} yield:
$$
V_1=\displaystyle{\exp\left(4\pi\frac{p_1}{p_2}\right)}-1,
$$
 and if $V_1=0$, then
$
V_2=4\pi s_1.
$
This completes the proof.
\end{proof}

\begin{nota}
We have $(V_1,V_2)=(0,0)$ iff $(p_1,s_1)=(0,0)$ i.e. equation \eqref{main equation} is Hamiltonian (see Theorem \ref{teo-Ham}).
Therefore, as the origin remains being monodromic, it is a center.
\end{nota}

Next we analyze the singular points of equation \eqref{polar1} such that $r\ne 0$. These are represented by the
non-zero critical points of the system under consideration.
\begin{lema}\label{lema solutions r theta}
Let $-\pi/12<\theta<\pi/12.$ Then the critical points of system \eqref{polar1}  with $r\ne 0$ are expressed by the following equalities:

\begin{equation}\label{solution r theta }
r={\frac{-p_2}{s_2+\sin\left(12\theta_{\pm}\right)},
\qquad\qquad
\theta_{\pm}=\frac{1}{6}\arctan(\Delta_{\pm})},
\end{equation}
where
$
\Delta_{\pm}=\displaystyle\frac{p_1\pm u}{p_2-p_1s_2+p_2s_1}$
and
$
u=\sqrt{p_1^2+p_2^2-\left(p_1s_2-p_2s_1\right)^2}$.
\end{lema}

\begin{proof}
Let $-\pi/12<\theta<\pi/12.$
To compute the equilibria of system \eqref{polar1}, we must find the roots of the following system:
\begin{equation}\label{critical points0}
    \begin{array}{l}
    \displaystyle{0=2rp_1+2r^2\left(s_1-\cos (12\theta)\right)}\\
    \displaystyle{0=p_2+r\left(s_2+\sin (12\theta)\right).}
    \end{array}
\end{equation}

Let $x=12\theta$ and $\displaystyle{t=\tan\frac{x}{2},}$ so $\displaystyle{t=\tan6\theta}.$ Then tedious but straightforward computations show that
\begin{equation}\label{Trigo}
\sin x=\displaystyle{\frac{2t}{1+t^2}}\qquad
\cos x=\displaystyle{\frac{1-t^2}{1+t^2}}.
\end{equation}
In the following step we remove $r$ from equations \eqref{critical points0}.
Making use of the preceding equations yield
\begin{equation*}
(-p_2+p_1s_2-p_2s_1)t^2+2p_1t+p_2+p_1s_2-p_2s_1=0.
\end{equation*}
Next we solve the previous equation for $t$ to get the result.

Finally, let us pay attention to the interval $-\pi/12\leqslant\theta<\pi/12.$ Let $x=12\theta$ and $\displaystyle{\tau=\cot\frac{x}{2}}.$ Therefore, we get $\tau=\cot6\theta.$
Expressions \eqref{Trigo} for $\sin x$ and $\cos x$ hold true also if $t$ is replaced by $\tau=1/t$. This completes the proof of the lemma.
\end{proof}

We will show now that simultaneous solutions of the type $(r,\theta)=(r,0)$, $(r,\theta)=(\tilde{r},\pi/12)$
are not possible.

\begin{lema}
Assume $|s_2|>1.$ We claim that there exist no parameter values which lead to critical points
of system \eqref{polar1} for $\theta=0$ and $\theta=\pi/12$ different that $(0,0).$
\end{lema}
\begin{proof}
We solve the following system
\begin{equation*}
\left\{
\begin{array}{l}
0=p_1+r(s_1-\cos(12\theta))\\
0=p_2+r(s_2+\sin(12\theta))
\end{array}
\right.
\end{equation*}
for $\theta=0,$ which yields $p_1s_2=p_2(s_1-1)$ with the additional constraints on the parameters
$\sign p_2=- \sign s_2$. This way, we have $r$ well defined in the
second equation. Secondly,
we solve the system for $\theta=\pi/12,$ which yields $p_1s_2=p_2(s_1+1).$ We have to put the same constraints on the parameters $\sign p_2=- \sign s_2.$
in consequence $p_2=0;$ therefore $r=0.$ This completes the proof.
\end{proof}

At this point, it is useful bring together the conditions in the parameter space that allow system \eqref{polar1} to have
one, thirteen or twenty--five critical points.

\begin{lema}\label{7critical points}
Suppose $s_2>1$ in system \eqref{polar1}. Then if $p_2\ge 0,$ the only critical point is $(0,0)$.
Assume $p_2<0.$ In this case the number of critical points is conditioned by the quadratic form:
\begin{equation}\label{quadraticForm}
{\mathcal Q}(p_1,p_2)=p_1^2+p_2^2-(p_1 s_2-p_2 s_1)^2=(1-s_2^2)p_1^2+(1-s_1^2)p_2^2+2s_1s_2p_1p_2
\end{equation}
Specifically:
\begin{enumerate}
\item \label{oneEq}
one critical point $(0,0)$ if ${\mathcal Q}(p_1,p_2)<0$;
\item \label{sevenEq}
thirteen critical points (origin plus one saddle-node corresponding to $\displaystyle{\frac{1}{12}}$ of the phase space), if ${\mathcal Q}(p_1,p_2)=0$;
\item \label{thirteenEq}
twenty--five critical points (origin plus a saddle and a node corresponding to $\displaystyle{\frac{1}{12}}$ of the phase space), if ${\mathcal Q}(p_1,p_2)>0$.
\end{enumerate}
Similarly, we obtain a corresponding result to $s_2<-1$ by switching $p_2>0$ by $p_2<0.$
\end{lema}

\begin{proof}
For brevity we perform the proof only for $s_2>1;$  the other cases are left as an easy exercise for the reader.

System \eqref{polar1} has always in origin one of its critical points. There will be more of those iff equations \eqref{solution r theta } have solutions for $r>0$. This is not the case in the following two situations. Firstly, if $r$ is negative, because in this case $p_2\ge 0$. Secondly if $\Delta_{\pm}<0;$ this is so because in this case ${\mathcal Q}(p_1,p_2)<0$.

In order for system \eqref{polar1} has exactly thirteen equilibria, we need two conditions. Firstly $p_2<0$ such that
$r>0$ as $s_2>1$. Secondly $\Delta_{+}=\Delta_{-}.$ But this last condition leads to
the discriminant $u=0$ in Lemma \ref{lema solutions r theta}.
Therefore, we get $r_+=r_-,~\theta_+=\theta_-$.

To show that these twelve equilibria are saddle-nodes, we work in a $1/12$ of the phase--space. his is due to the fact that system \eqref{polar1} is $\mathbb{Z}_{12}$--equivariant.

The first step is to evaluate the Jacobian matrix of system \eqref{polar1} at the critical points $(r_+,\theta_+).$ The Jacobian matrix $J$ is
\begin{equation}\label{jacobian general}
J_{(r_+,\theta_+)}=
\left(\begin{matrix}
            2p_1+4r_+(s_1-\cos(12\theta_+)&24r_+^2\sin(12\theta_+)\\
            s_2+\sin(12\theta_+)&12r_+\cos(12\theta_+)\\
        \end{matrix}\right).
\end{equation}

Next we evaluate $J$ at the twelve equilibria. Taking into account the constraint $p_1^2+p_2^2=(p_1s_2-p_2s_1)^2,$ the eigenvalues of $J$ are
$$
\lambda_1=0
\qquad\qquad
\lambda_2=\displaystyle{-\frac{2(p_1^2+6p_2^2 +p_1^2s_1+6p_2^2s_1-5p_1p_2s_2+p_1p_2s_1s_2-p_1^2s_2^2)}{p_1+p_1s_1+p_2s_2+p_2s_1s_2-p_1s_2^2}}.
$$

Hence, $(r_+,\theta_+)$ has a eigenvalue of value zero. Next we want to prove that these equilibria are saddle-nodes.
We know from \cite{Andronov}, that the sum of the indices of the equilibria inside the limit cycle equals $+1.$ We already proved that infinity has no critical points. Therefore, we deal with a limit cycle of a planar system which has thirteen equilibria inside. One of them is the origin; we have proved that it is a focus and therefore it has index +1. Moreover, there are $12$ more equilibria, all of the same kind due to the $\mathbb{Z}_{12}$--symmetry. Therefore, all of these additional equilibria must have zero index. Since we have showed
that these equilibria are semi-hyperbolic, it follows that they are saddle-nodes.

The last task in our mission to prove this results consists in deriving the conditions in the parameter space such that equation \eqref{polar1} has exactly twenty--five equilibria. This is easily shown by observing that $r>0$ ({\it i.e.} $p_2<0$) and
the discriminant in $\Delta_{\pm}>0$ of Lemma~\ref{lema solutions r theta}.

Next we evaluate the Jacobian matrix \eqref{jacobian general} at the twenty--four equilibria to analyze their stability. Putting the condition
$p_1^2+p_2^2>(p_1s_2-p_2s_1)^2,$ the eigenvalues of the critical points
$(r_+,\theta_+)$ are $\lambda_{1,2}=R_+\pm S_+,$
while the eigenvalues of $(r_-,\theta_-)$ are $\alpha_{1,2}=R_-\pm S_-,$ where
\begin{equation*}
\begin{array}{l}
R_{\pm}=\displaystyle{\frac{\pm p_1^2s_1\pm5p_2^2s_1\pm5p_1p_2s_2+4p_1u}
{\pm p_1s_1\pm p_2s_2+u}}+\\
\\
\hspace{1cm}\displaystyle{\frac{\frac{1}{2}\sqrt{96(p_1^2+p_2^2)u(\pm p_1s_1\pm p_2s_2+u)+4\left(p_1^2s_1+6p_2^2s_1-5p_1p_2s_2\pm p_1u\right)^2}}
{\pm p_1s_1\pm p_2s_2+u}}\\
\\
S_{\pm}=\displaystyle{\frac{\pm p_1^2s_1\pm6p_2^2s_1\pm5p_1p_2s_2+5p_1u}
{\pm p_1s_1\pm p_2s_2+u}}-\\
\\
\hspace{1cm}\displaystyle{\frac{\frac{1}{2}\sqrt{96(p_1^2+p_2^2)u(\pm p_1s_1\pm p_2s_2+u)+4\left(p_1^2s_1+6p_2^2s_1-5p_1p_2s_2\pm5p_1u\right)^2}}
{\pm p_1s_1\pm p_2s_2+u}}\\.
\end{array}
\end{equation*}

Our final task is proving that one of these equilibria has index $+1,$ and the other is a saddle.

Some tedious but straightforward computations show that the product of the eigenvalues of $(r_+,\theta_+)$ and $(r_-,\theta_-)$ are, respectively
\begin{equation*}
\begin{array}{l}
R_+^2-S_+^2=\displaystyle{\frac{2(p_1^2s_1+6p_2^2s_1-5p_1(p_2s_2-u))}
{(p_1s_1+p_2s_2-u)^2}}\cdot        \\
\\
\hspace{2cm}\displaystyle{\sqrt{96(p_1^2+p_2^2)u(-p_1s_1-p_2s_2- u)+4\left(p_1^2s_1+6p_2^2s_1+5p_1(-p_2s_2+u)\right)^2}
} \\
\\
R_-^2-S_-^2=\displaystyle{-\frac{2(p_1^2+6p_2^2s_1-5p_1(p_2s_2+u))}
{(p_1s_1+p_2s_2+ u)^2}} \cdot        \\
\\
\hspace{2cm}\displaystyle{\sqrt{96(p_1^2+p_2^2)u(p_1s_1+p_2s_2+ u)+4\left(p_1^2s_1+6p_2^2s_1-5p_1( p_2s_2+u)\right)^2}
} \\
\end{array}
\end{equation*}
The denominator of these formulae is positive.

If $p_1<0$ and $\displaystyle{u<\frac{-p_1^2s_1-5p_2^2s_1+5p_1p_2s_2}{5p_1}}$ or $p_1>0$ and $\displaystyle{u>\frac{-p_1^2s_1-6p_2^2s_1+5p_1p_2s_2}{5p_1}}$ then $R_-^2-S_-^2<0$. Therefore the equilibrium $(r_-,\theta_-)$ is a saddle.
If $p_1<0$ and $\displaystyle{u>\frac{p_1^2s_1+6p_2^2s_1-5p_1p_2s_2}{5p_1}}$ or $p_1>0$ and $\displaystyle{u<\frac{p_1^2s_1+6p_2^2s_1-5p_1p_2s_2}{5p_1}}$ then $R_-^2-S_-^2>0$. Hence, $(r_+,\theta_+)$ has index $+1.$ This completes the proof.
\end{proof}

\begin{lema}\label{signA}
Let us assume $|s_2|>1$ and consider the following quantities:
\begin{equation*}
\Sigma_A^-=\frac{p_2s_1s_2- \sqrt{p_2^2\left(s_1^2+s_2^2-1\right)}}{s_2^2-1},
\hspace{0.3cm}\Sigma_A^+=\frac{p_2s_1s_2+ \sqrt{p_2^2\left(s_1^2+s_2^2-1\right)}}{s_2^2-1}.
\end{equation*}
We denote $A(\theta)$ the function given in Lemma \ref{lema abelian}. It follows that $A(\theta)$ changes sign iff
$p_1\in\left(\Sigma_A^-,\Sigma_A^+\right).$
\end{lema}
\begin{proof}

Let $x=\sin(12\theta),~y=\cos(12\theta)$. Then $A(\theta)$ defined in \eqref{abelian1} writes
\begin{equation*}
A(x,y)=\displaystyle{\frac{2}{p_2}\left(p_1-p_2s_1s_2+p_1s_2^2+(2p_1s_2-p_2s_1)x+  (p_2x-p_1y+p_2s_2)y\right)}.
 \end{equation*}
Next we solve the nonlinear system
\begin{eqnarray*}
A(x,y)=0,\\
x^2+y^2=1.
\end{eqnarray*}
which yields the solutions
\begin{equation*}
\begin{array}{l}
x_1=-s_2,~y_1=\sqrt{1-s_2^2}\\
x_2=-s_2,~y_2=-\sqrt{1-s_2^2}\\\\
x_\pm=\displaystyle{\frac{p_1p_2s_1-p_1^2s_2\pm p_2\sqrt{p_1^2+p_2^2-(p_2s_1-p_1s_2)^2}}{p_1^2+p_2^2}}\\
\\
y_\pm=\displaystyle{\frac{p_2^2s_1-p_1p_2s_2\mp p_1 \sqrt{p_1^2+p_2^2-(p_2s_1-p_1s_2)^2}}{p_1^2+p_2^2}}\\
\end{array}
\end{equation*}
Since $x=\sin(12\theta)=-s_2<-1,$ it follows that two pairs of solutions $(x_1,y_1), (x_2,y_2)$ cannot be solutions $A(\theta)=0.$

On the other hand, we have two choices for the intervals for which $A(\theta)$ doesn't change sign:
\begin{enumerate}
\item{} $|x_\pm|>1$ (hence $|y_\pm|>1$)
\item{} $\Delta=p_1^2+p_2^2-(p_2s_1-p_1s_2)^2\leq0.$
\end{enumerate}
If $\Delta<0$, we get complex non-real solutions, while if $\Delta=0$, the function has a double zero; in this case it doesn't
change sign.

Finally, while $|x_\pm|>1$ is impossible in the parameter space subject to our mentioned constraints, the second case yields the region $p_1\in\mathcal{R}\setminus (\Sigma_A^+,\Sigma_A^-).$
\end{proof}

\begin{lema}\label{sigmalemaB}
Let us assume $|s_2|>1$ and consider the following quantities:
\begin{equation*}
\Sigma_B^\pm=\frac{\Sigma_A^\pm}{2}.
\end{equation*}
We denote $B(\theta)$ be the function given in Lemma \ref{lema abelian}. It follows that $B(\theta)$ changes sign iff
$p_1\in\left(\Sigma_B^-,\Sigma_B^+\right).$
\end{lema}

\begin{proof}
We carry out the proof by performing calculations similar to the ones in the proof of the previous lemma. We denote $x=\sin(12\theta), y=\cos(12\theta),$
and calculate the roots of the nonlinear system
\begin{eqnarray*}
B(x,y)=0,\\
x^2+y^2=1.
\end{eqnarray*}
These are
\begin{eqnarray*}
x_\pm= \frac{2p_1p_2s_1-4p_1^2s_2\pm p_2\sqrt{4p_1^2+p_2^2-(p_2s_1-2p_1s_2)^2}}{4 p_1^2+p_2^2},\\
y_{\pm}=\frac{p_2^2s_1-2p_1p_2s_2\mp 2p_1\sqrt{4p_1^2+p_2^2-(p_2s_1-2p_1s_2)^2}}{4 p_1^2+p_2^2}.
\end{eqnarray*}
We get that $B(\theta)$ doesn't change sign iff
$p_1\not\in(\Sigma_B^+,\Sigma_B^-).$
\end{proof}

\begin{nota}\begin{itemize}
\item we can find examples of systems are such that $A(\theta)$ changes sign but $B(\theta)$ does not
and vice-versa.
\item The constraint that forces function $A(\theta)$ not to change sign is related to the number of equilibria of system \ref{polar1}.
Condition $\Delta<0$ in the proof of Lemma \ref{signA} coincides with the constraint for the existence of a unique equilibrium in
system \eqref{polar1}. Moreover, the constraint $\Delta=0$ forces the system to have exactly $13$ critical points, together with $p_2<0.$
\end{itemize}
\end{nota}
Next we will recall a result on scalar Abel equations due to Llibre \cite{Llibre}.

\begin{teo}(Llibre, \cite{Llibre})\label{teorema Llibre}
Consider the Abel equation \eqref{abelian0} and assume that either $A(\theta)\not\equiv0$ or $B(\theta)\not\equiv0$ does not change sign.
Then it has at most three solutions satisfying $x(0)=x(2\pi),$ taking into account their multiplicities.
\end{teo}

\section{Limit cycles}\label{Proof of Theorem $1.$}
\begin{proof}[Proof of Theorem~\ref{teorema principal}]
Let us begin our proof by defining the function $c(\theta)=s_2+\sin(12\theta)$ and the set $\Theta:=\{(r,\theta):\dot{\theta}=p_2+(s_2+\sin(12\theta))r=0\}.$
We have $c(\theta)\neq0,~\forall\theta\in[0,2\pi].$ This is so because $|s_2|>1.$

\begin{figure}[ht]
\begin{center}
\includegraphics[scale=0.45]{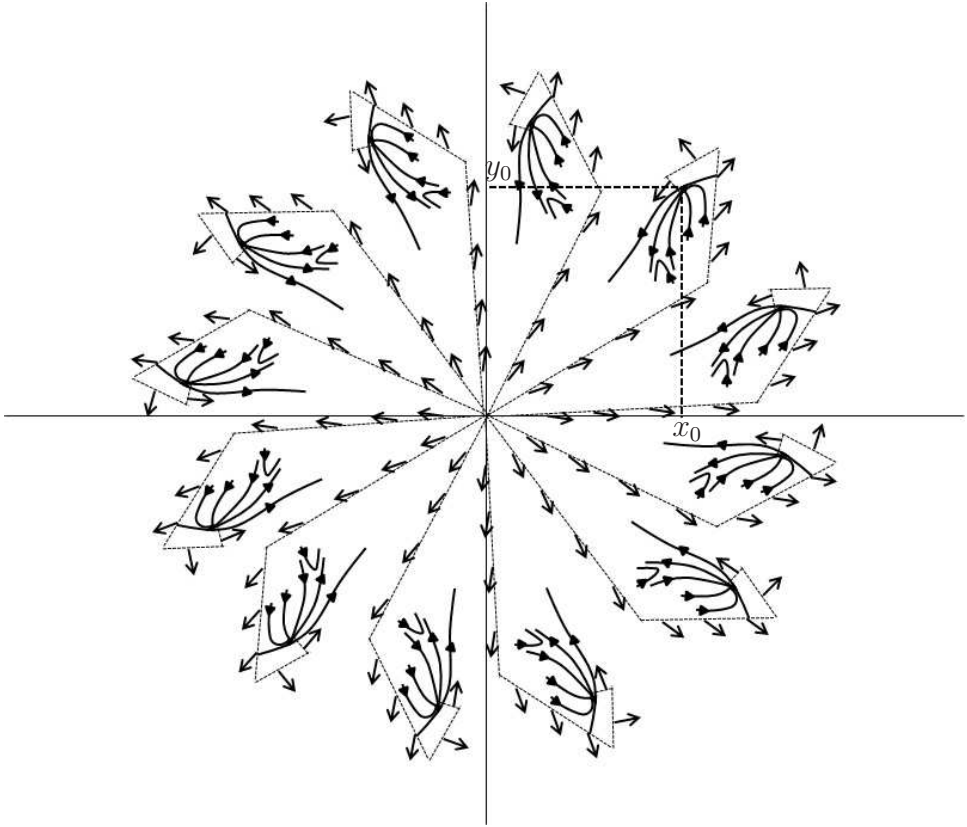}\caption{The polygonal curve with no contact with the flow of the differential equation and the separatrices of the saddle-nodes of system \eqref{critical points0}.}\label{figure2}
\begin{picture}(0,0)\vspace{0.5cm}
\put(71,188){\large{$x_0$}}
\put(0,287){\large{$y_0$}}
\end{picture}
\end{center}
\end{figure}

$(a)$ To prove the first item of the Theorem~\ref{teorema principal}, we assume that condition $(i)$ is satisfied. By applying Lemma \ref{lema abelian}, it is enough to study the non contractible periodic orbits of the Abel equation \eqref{abelian0}.
We have that $A(\theta)$ doesn't change sign. This is so because $p_1\notin\left(\Sigma_A^-,\Sigma_A^+\right),$ by condition $(i)$ of Lemma \ref{signA}. Moreover, the upper bound for the solutions satisfying $x(0)=x(2\pi)$ in system \eqref{abelian0} taking into account their multiplicities, is three. This results directly by applying Llibre's Theorem \ref{teorema Llibre}. One of these solutions is $x=0.$
In addition, the curve $x=1/c(\theta)$ is a second solution satisfying this condition, since $c(\theta)\neq0.$ It is easy to show that the equation $x=1/c(\theta)$ is mapped into infinity of the differential equation. One possibility do do this is by undoing the Cherkas transformation. Therefore, by applying
Lemma \ref{lema abelian}, the upper bound for the limit cycles of equation \eqref{main equation} is one. The same
lemma proves that it is hyperbolic.

$(b)$ To prove the second item of Theorem~\ref{teorema principal} we fix parameters $p_2=-1,~s_1=-0.5,~s_2=1.2.$
We get $\Sigma_A^-=-0.52423$, $\Sigma_A^+=3.25151.$ If $p_1>\Sigma_A^+,$ by Lemmas \ref{lema characterization origin}  and \ref{lema Lyapunov}, both the origin and infinity have positive sign in the Poincar\'e compactification. More specifically the origin is an unstable focus.
By applying Lemma \ref{signA}, the function $A(\theta)$ still exists for $p_1>\Sigma_A^+.$ Under these circumstances the origin is the unique critical point. Now we have $p_2A(\theta)>0.$ Therefore the exterior of the closed curve $\Theta$ is positively invariant. To prove that there is exactly one hyperbolic limit cycle surrounding the curve $\Theta,$ we apply the Poincar\'e-Bendixson Theorem. In addition it is stable.

for $p_1=\Sigma_A^+,$ twelve more semi-elementary additional equilibria appear in the phase space. They are located on $\Theta.$ We have showed in Lemma \ref{7critical points} that these equilibria are saddle-nodes. our strategy is now to show that for this $p_1$ the periodic orbit surrounds the thirteen critical points still exists. our tool will be a polygonal line with no contact with the flow of the differential equation. On this polygonal line, the vector field points outside, and
since the infinity is repellor, the $\omega$-limit set of the unstable separatrices of the saddle-nodes has to be a limit cycle surrounding the curve $\Theta.$ This polygonal line with no contact with the flow is shown in Figure \ref{figure2}.

Again, as in many occasions occurred throughout this paper we use the $\mathbb{Z}_{12}-$equivariance of system \eqref{main equation}. It gives us the possibility of studying the flow in just $\displaystyle{\frac{1}{12}}$ of the phase space.

We build the polygonal line by working in cartesian coordinates. We defined the following polygonal line.
\begin{equation*}
(x(t),y(t))=
\left\{
\begin{array}{l}
\displaystyle{\left(t,t\right)\hspace{7.4cm}\mathrm{if}~0\leqslant x\leqslant \sqrt{\frac{1}{2.4}}},\\
\\
\displaystyle{\left(t,1.5\left(t-\sqrt{\frac{1}{2.4}}\right)+\sqrt{\frac{1}{2.4}}\right)
\hspace{3.2cm}\mathrm{if}~\sqrt{\frac{1}{2.4}}\leqslant t<1.4},\\
\\
\displaystyle{(-0.1(t-1.4)+1.4,1.8)}
\displaystyle{\hspace{5cm}\mathrm{if}~1.4\leqslant t<2},\\
\\
\displaystyle{((t-3)(x_0-1)+x_0,(t-3)(y_0-1.8)+y_0)}
\displaystyle{\hspace{2.3cm}\mathrm{if}~2\leqslant t<3}.
\end{array}
\right.
\end{equation*}
The flow associated to the differential equation is transversal to the polygonal line. We get that when evaluated on the corresponding segments of the polygonal line, the scalar product between the normal to each segment and the flow of this equation is negative.
To show the calculation, we will exemplify it for the first segment.

Let's take the segment $L\equiv\displaystyle{\{y=x\}}$ and when substituting it into the system in cartesian coordinates \eqref{cartesian} we get
\begin{equation*}
\left\{
\begin{array}{l}
\dot{x}=\displaystyle{16 p_1 x^9 - 16 p_2 x^9 + 32 x^{11} + 32 s_1 x^{11} - 32 s_2 x^{11}}\\
\\
\dot{y}=\displaystyle{16 p_1 x^9 + 16 p_2 x^9 + 32 x^{11} + 32 s_1 x^{11} + 32 s_2 x^{11}}.
\end{array}
\right.
\end{equation*}
A normal vector to the line $L$ is $(-1,1)$ and the scalar product of $(\dot{x},\dot{y})$ with $(-1,1)$ yields $\displaystyle{f(x)=32 p_2 x^9 + 64 s_2 x^{11}}.$
Solving this last equation leads to $f(x)<0$ for $\displaystyle{-\sqrt{\frac{-p_2}{2s_2}}\leqslant x\leqslant \sqrt{\frac{-p_2}{2s_2}}}.$
So we choose $\displaystyle{0\leqslant x\leqslant \sqrt{\frac{-p_2}{2s_2}}},$ and we get that this scalar product is negative in the region where the
polygonal line is defined as $(t,t).$ The remaining segments of the polygonal line are calculated similarly. We denote $(x_0,y_0)=(1.19,1.63)$ the coordinates point of the saddle node represented in Figure \ref{figure2}.

Next we combine the fact that infinity is repellor with the Poincar\'e-Bendixson Theorem to conclude that
the only possible $\omega-$limit for the unstable separatrix of the saddle-node is a periodic orbit; moreover, this has to surround the twelve saddle-nodes, see again Figure \ref{figure2}.

For $p_1=\Sigma_A^+$ then $A(\theta)$ keeps its sign. Using similar arguments as in part $(a)$ of the proof enables us to conclude that the limit cycle is hyperbolic.

If we slightly change $p_1$ towards zero but still very close to $\Sigma_A^+,$ it follows that $B(\theta)$ is positive because $\Sigma_A^+>\Sigma_B^+;$ therefore, there are $25$ critical points as shown in Lemma \ref{7critical points}: the origin (which is a focus), twelve saddles and twelve critical points of index $+1$ on $\Theta.$ Next we apply one more time part $(a)$ of the proof of this theorem, we know that the upper bound for the limit cycles surrounding the origin is one.
For $p_1=\Sigma_A^+$ the limit cycle is hyperbolic and it still. Hence we get that the vector field corresponding to $B(\theta)$ of constant sign, has $25$ non-zero critical points and a limit cycle which surrounds them together with the origin.
\end{proof}

\paragraph{\bf Acknowledgements}
I would like to thank my parents for supporting me during my unemployment period. I also thank Professor Bard Ermentrout for his kind help with the XPPAUT software.

\end{document}